\documentclass[11pt]{article}
\usepackage{latexsym,amsmath,amsthm,verbatim,ifthen,amssymb,graphicx,color,mathrsfs}
\usepackage[all]{xy}
\usepackage{color}

\SelectTips{cm}{}
\newtheorem{theorem}{Theorem}[section]

\newtheorem{corollary}[theorem]{Corollary}
 \newtheorem{lemma}[theorem]{Lemma}
 
 \theoremstyle{definition}
 
 \theoremstyle{remark}
 \newtheorem{remark}[theorem]{Remark}
 
\DeclareRobustCommand{\smalltree}{\ \ \begin{picture}(0,0)(0,-5)
\qbezier(0,-3)(2,-1)(3,0)
\qbezier(-3,0)(-2,-1)(0,-3)
\qbezier(0,-3)(0,-3)(0,-6)
\end{picture}\ \ }

\def\bracket{Q}
\def\Aut{\operatorname{Aut}}

\newcommand{\bz}{\mathbb Z}

\def\ad{{\rm ad}}
\def\id{{\rm id}}

    \newcommand{\lasu}{{\mathfrak{L}}}
    \newcommand{\cotimes}{{\widehat{\otimes}}}
    \newcommand{\asu}{{\mathscr{A}}}

    \newcommand{\Qfuntor}{\mathcal{L}}

 \newcommand{\lib }{\mathbb{L}}

\newcommand{\cdgl}{\operatorname{{\text{\rm cDGL}}}}

\newcommand{\Hom}{\operatorname{\text{\rm Hom}}}

 \newcommand{\MC}{\operatorname{{\rm MC}}}
\newcommand{\mc}{{\MC}}

\newcommand{\bk}{\mathbb K}

 \def\Aut{\operatorname{Aut}}
   \newcommand{\libc}{{\widehat\lib}}

\begin{document}

\title{The infinity Quillen functor, Maurer-Cartan elements and DGL realizations}
\author{Urtzi Buijs, Yves F\'elix, Aniceto Murillo and Daniel Tanr\'e\footnote{The  authors have been partially supported by the MINECO grants MTM2013-41768-P and MTM2016-78647-P. The first author has also been partially supported by  the Ram\'on y Cajal MINECO program.  The first and third authors have also been partially supported by the Junta de Andaluc{\'\i}a grant FQM-213. \vskip 1pt {\em 2010 Mathematics Subject
Classification}. Primary: 55P62, 17B55; Secondary: 55U10.\vskip
 1pt
 {\em Key words and phrases}: Realization of Lie algebras. Cosimplicial Lie algebras. Maurer-Cartan elements. Quillen functor. Rational homotopy theory}}

\maketitle

\begin{abstract}
We show an alternative construction of the cosimplicial free complete diferential graded Lie algebra $\lasu_\bullet=\libc(s^{-1}\Delta^\bullet)$ based on a new Lie bracket formulae for Lie polynomials on a general tensor algebra.  Based on it,we prove that for any complete differential graded Lie algebra $L$, its geometrical realization $\langle L\rangle=\Hom_{\cdgl}(\lasu_\bullet,L)$ is isomorphic to its nerve $\gamma_\bullet(L)$, a deformation retract of  the Getzler-Hinich realization $\mc(\asu_\bullet\cotimes L)$.
\end{abstract}

\section*{Introduction}
In \cite{pri} we introduced a cosimplicial complete differential graded Lie algebra (cDGL henceforth) $\lasu_{\bullet}$ by extending the Lawrence-Sullivan interval \cite{lawsu} to any simplex: for each $n\ge 1$, $\lasu_n=\libc(s^{-1}\Delta^n)$ is the free cDGL  in which $s^{-1}\Delta^n$ together with the linear part of the differential is the (desuspension) of the simplicial chain complex of the standard $n$-simplex, and the vertices correspond to  Maurer-Cartan elements. The existence of such a cosimplicial cDGL satisfying the above conditions can be deduced in different ways. In \cite[Theorem 3.3 and 3.4]{pri} it was constructed by an inductive procedure. In \cite[Theorem 7.3 (i)]{pri} and \cite[\S 5.2]{ni1} it was shown how to construct it using the homotopy transfer theorem and Lie coalgebras. Nevertheless, all of them are isomorphic thanks to the unicity theorem \cite[Theorem 2.8]{pri}.

In this paper we give an alternative construction of $\lasu_\bullet$ based on a new Lie bracket formulae for Lie polynomials on a general tensor algebra (see Theorem \ref{Liepolynomial}), in the spirit of \cite[Chapter 1]{reu}. This formula allows us to extend the classical Quillen functor $\Qfuntor $ to $C_\infty$-coalgebras (see Theorem \ref{Quillenconstruction}).

The cosimplicial cDGL $\lasu_\bullet$ let us geometrically realize any cDGL $L$ as the simplicial set
$$\langle L\rangle=\Hom_{\cdgl}(\lasu_\bullet,L).$$
We also proved \cite[\S8]{pri} that under usual bounding and finite type assumptions this simplicial set is homotopy equivalent to any other known realization of a cDGL $L$, i.e., the $\lambda$-Quillen functor on $L$ \cite{qui}, the realization of the Chevalley-Eilenberg cochain functor on $L$ \cite{bousgu}, and the Getzler-Hinich simplicial functor on $L$ \cite{getz,hi}. The latter, carefully studied  also in \cite{ber}, is defined as the simplicial set of Maurer-Cartan elements of the simplicial DGL $\asu_\bullet\cotimes L$ in which $\asu_\bullet$ denotes the simplicial commutative differential graded algebra  of PL-differential forms on the standard simplices \cite{du,su} and $
A\cotimes L=\varprojlim_n (A\otimes L/L^n)$ for any commutative differential graded algebra (CDGA henceforth). In this paper, we show that, with full generality, this is homotopy equivalent to our realization:

\begin{theorem}\label{principal}  For any $\cdgl$ $L$ there are explicit homotopy equivalences
$$
\xymatrix{  \mc(\asu_\bullet\cotimes L)
\ar@<0.75ex>[r]^-p & \langle L\rangle \ar@<0.75ex>[l]^-i }
$$
which make  $\langle L\rangle$ a strong homotopy retract of $\mc(\asu_\bullet\cotimes L)$.
\end{theorem}
As a consequence we obtain (see Theorem \ref{final}):

 \begin{theorem}\label{corolario}  The realization $\langle L\rangle$ of any cDGL is isomorphic to $\gamma_\bullet(L)$ as simplicial sets.
  \end{theorem}
  The  {\em nerve} $\gamma_\bullet(L)$  of $L$, introduced  in \cite[\S5]{getz},  is a subsimplicial set of $\mc(\asu_\bullet\cotimes L)$ homotopy equivalent to it.

 To prove these results we need to add  two new entries to the so called ``Rosetta Stone'' of higher structures \cite[\S10.1.9]{LV}: any homotopy retract (see next section for precise definitions)
$$
\xymatrix{ \ar@(ul,dl)@<-5.6ex>[]  & C
\ar@<0.75ex>[r] & V \ar@<0.75ex>[l] }
$$
of the cocommutative differential graded coalgebra (CDGC) $C$, in which $V$ is finite type and bounded below, produces a free $\cdgl$ $\Qfuntor(V)=\libc(s^{-1}V)= \varprojlim_n \lib(s^{-1}V)/\lib^{>n}(s^{-1}V)$ and  a homotopy retract
$$
\xymatrix{ \ar@(ul,dl)@<-5.6ex>[]  & \Hom(C,L)
\ar@<0.75ex>[r] & \Hom(V,L) \ar@<0.75ex>[l] }
$$
of the convolution Lie algebra $\Hom(C,L)$ for any DGL $L$. This induces an $L_\infty$-algebra structure on $\Hom(V,L)$ for which the  following holds (see Theorem \ref{main}).

\begin{theorem}\label{main1intro} For any $\cdgl$ $L$, there is a bijection
$$
\mc\Hom(V,L)\cong{\Hom_{\cdgl}}(\Qfuntor(V),L).
$$
\end{theorem}
The dual result is also attained although finite type requirements are essential: a homotopy retract as above, in which now $C$ is a commutative differential graded algebra  and $V$ is of finite type, also provides a free cDGL $\lasu(V^\sharp)$ and a homotopy retract
$$
\xymatrix{ \ar@(ul,dl)@<-5.6ex>[]  & C\cotimes L
\ar@<0.75ex>[r] & V\cotimes L=V\otimes L \ar@<0.75ex>[l] }
$$
of the  Lie algebra $C\cotimes L$ for any cDGL $L$. This induces an $L_\infty$-algebra structure on $V\otimes L$ which is naturally isomorphic to $\Hom(V^\sharp,L)$. Then (see Theorem \ref{main2}):

\begin{theorem}\label{main2intro} For any $\cdgl$ $L$ there is a bijection
$$
\mc(V\otimes L)\cong{\Hom_{\cdgl}}(\Qfuntor(V^\sharp),L).
$$
\end{theorem}
These results can be extended to the general case in which $C$ of the homotopy retract above is either a $C_\infty$-algebra or $C_\infty$-coalgebra. Nevertheless, for our purposes we only need them as stated.

Next section contains notation and recalls explicit descriptions of transferred structures. In Section \ref{Quillen} we prove Theorem \ref{Liepolynomial} and Theorem \ref{Quillenconstruction}. Section \ref{realization} contains Theorems \ref{principal} , \ref{corolario}, \ref{main1intro} and \ref{main2intro}.
\smallskip

 Theorems $\ref{principal},\ref{main1intro}$ and $\ref{main2intro}$  have been independently and simultaneously obtained by D. Robert-Nicoud in \cite{ni1,ni2} using different methods.

\section{Preliminaries} \label{Introduction}

With the aim of fixing
notation, we begin by recalling known facts and setting some generic assumptions on the algebraic structures of our concern.
Abusing notation, we will not distinguish a given category $\mathcal C$ from the class of its objects. Its morphism sets are denoted by $\Hom_{\mathcal C}$ and unadorned $\Hom$ denotes just linear maps.  The coefficient field for any algebraic object $\bk$ is assumed to be of characteristic zero. Any graded object is considered $\bz$-graded and under no finite type assumptions unless explicitly specified otherwise.

An \emph{$A_{\infty}$-algebra} is a graded vector space $A$ endowed with a family of linear maps,
 $$
 m_k\colon A^{\otimes k}\longrightarrow A,\quad, k\ge1,
 $$
 of degree $k-2$, such that for all $i\ge 1$,
 $$
\sum_{k=1}^i\, \sum_{n=0}^{i-k}(-1)^{k+n+kn}m_{i-k+1}({\rm id}^n\otimes m_k\otimes{\rm id}^{\otimes i-k-n})
=0.
$$
An $A_\infty$-algebra $A$ is {\em commutative}, or it is a {\em $C_\infty$-algebra} if, for each $k\ge2$, the $k$-th multiplication $m_k$ vanishes on the {\em shuffle products}, that is, $m_k\nu_k=0$, with
$$\nu_k\colon A^{\otimes k}\to A^{\otimes k},\quad \nu_k(a_1\otimes\dots\otimes a_k)=\sum_{i=1}^k\sum_{\sigma\in S(i,k-i)}\varepsilon_\sigma\, a_{\sigma(1)}\otimes\dots\otimes a_{\sigma^(k)},
$$
where $S(i,k-i)$ denotes
the set of $(i, k-i)$-shuffles, i.e., permutations $\sigma$
such that $\sigma^{-1}(1)<\cdots<\sigma^{-1}(i)$ and $\sigma^{-1}(i+1)<\cdots<\sigma^{-1}(k)$.

A {\em  differential graded algebra}, DGA henceforth, is an $A_\infty$-algebra for which $m_k=0$ for all $k\ge 3$. In this case $m_1=d$ and $m_2=m$ are the {\em differential} and {\em multiplication} respectively.

Dually, an \emph{$A_{\infty}$-coalgebra} is a graded vector space $C$ endowed with a family of linear maps,
 $$
 \Delta_k\colon C\longrightarrow C^{\otimes k},\quad, k\ge1,
 $$
 of degree $k-2$, such that for all $i\ge 1$,
 $$
\sum_{k=1}^i\, \sum_{n=0}^{i-k}(-1)^{k+n+kn}({\rm id}^{\otimes i-k-n}\otimes \Delta_k
\otimes {\rm id}^n)\Delta_{i-k+1}=0.
$$
An $A_{\infty}$-coalgebra $C$ is \emph{cocommutative}, or it is a {\em $C_\infty$-coalgebra} if, for each $k\ge 2$, the {\em unshuffle products} vanish on the image of the $k$-th comultiplication $\Delta_k$,  that is, $\tau\circ\Delta_k=0$ with
$$\tau\colon C^{\otimes k}\to C^{\otimes k},\quad
\tau(c_1\otimes\cdots\otimes c_k)=\sum_{i=1}^k\sum_{\sigma\in S(i, k-i)}
\varepsilon_{\sigma}\, c_{\sigma^{-1}(1)}\otimes\cdots\otimes  c_{\sigma^{-1}(k)}.
$$
A {\em  differential graded coalgebra}, DGC henceforth (CDGC if it is cocommutative), is an $A_\infty$-coalgebra for which $\Delta_k=0$ for all $k\ge 3$. In this case $\Delta_1=\delta$ and $\Delta_2=\Delta$ are the {\em codifferential} and {\em comultiplication} respectively.

An \emph{$L_{\infty}$-algebra} is a
graded vector space $L$ together with linear maps $\ell_k\colon
L^{\otimes k}\to L$ of degree $k-2$, for $k\ge 1$, satisfying:
\smallskip

(i) For any permutation $\sigma$ of $n$ elements, and any $n$-tuple $x_1,\dots,x_n$ of homogeneous elements of $L$,
$$
\ell_k(x_{\sigma(1)},\ldots, x_{\sigma(n)})=\varepsilon_{\sigma}\varepsilon\ell_k(x_1,\ldots, x_n),
$$
where $\varepsilon_{\sigma}$ is the signature of the permutation and
$\varepsilon$ is the sign given by the Koszul convention.
 \smallskip

 (ii) For any $n\ge 1$ and any $n$-tuple $x_1,\dots,x_n$ of homogeneous elements of $L$, the \emph{generalized Jacobi identity} holds,
$$
\sum_{i+j=n+1}\sum_{\sigma\in S(i, n-i)}\varepsilon_{\sigma}\varepsilon(-1)^{i(j-1)}
\ell_{n-i}(\ell_i(x_{\sigma(1)},...,x_{\sigma(i)}), x_{\sigma(i+1)},..., x_{\sigma(n)})=0,
$$
The set of {\em Maurer-Cartan elements} of an $L_\infty$-algebra $L$ is defined as
\begin{equation*}
\mc(L)=\{ x\in {L}_{-1}, \,\, \sum_{k\geq 0}\dfrac{1}{k!}\ell_k({x,\dots ,x})=0  \}.
\end{equation*}

A {\em differential graded Lie algebra}, DGL henceforth, is an $L_\infty$-algebra $L$ for which $\ell_k=0$ for all $k\ge 3$. In this case $\ell_1=\partial$ and $\ell_2=[\,\,,\,]$ are the {\em differential} and the {\em Lie bracket} respectively. In this case $\mc(L)=\{ x\in {L}_{-1} \ |\ \partial x=-\frac{1}{2}[x,x] \}$. A DGL $L$ is called \emph{free} if it is free as a Lie algebra, that is, $L = \mathbb L(V)$ for some graded vector space $V$. Recall that $\lib(V)\subset T(V)$ is the Lie algebra generated by commutators on $V$. A DGL $L$ is complete (cDGL henceforth) if $L=\varprojlim_n L/L^n$
where $L^1 = L$ and for $n\geq 2$, $L^n = [L, L^{n-1}]$.

Consider a diagram
$$
\xymatrix{ \ar@(ul,dl)@<-5.6ex>[]_K  & M
\ar@<0.75ex>[r]^-p & V \ar@<0.75ex>[l]^-i }
$$
where $i$ and $p$ are chain maps for which $pi=\id_V$ and $ip\simeq \id_M$ through the chain homotopy $K$. We encode this data as $(M,V,i,p,K)$ and call it a {\em homotopy retract}.

For such a retract, the
\emph{homotopy transfer theorem} \cite{Fuk03, KS00, KS01, LV, Mer99}, also known as the {\em homological perturbation lemma} \cite{GS86, GLS91, HK91, Kai83} permits to transfer any additional structure on $M$ to the corresponding infinity version on $V$:

\begin{theorem} \label{HTT} Given $(M,V,i,p,K)$ a homotopy retract in which $M$ is either a  DGA (resp. CDGA), DGC (resp. CDGC) or DGL,
there exists an  $A_\infty$-algebra (resp. $C_\infty$-algebra), $A_\infty$-coalgebra (resp. $C_\infty$-coalgebra) or $L_\infty$-algebra structure in $V$, unique up to isomorphism,
    and   quasi-isomorphisms in the corresponding infinity structure
$$
\xymatrix{ M
\ar@<0.75ex>[r]^-P & V \ar@<0.75ex>[l]^-I }
$$
 extending $i$ and $p$ and such that $PI=\id_V$.
\end{theorem}
The general fact that CDGA's transfer to $C_\infty$-algebras  is proved in \cite[Theorem 12]{ChG}. A dual argument proves the analogue for commutative DGC's.

As we shall strongly use it, we describe in each case, the explicit description of the transferred structure in the above theorem. In what follows, for any $k\ge 2$, $\mathscr{PT}_k$ denote the set of isomorphism classes of planar rooted binary trees of $k$ leaves, while $\mathscr{T}_k$ consists of isomorphism classes of (non planar) rooted binary trees with $k$ leaves.

Let first $M=(C,\delta,\Delta)$ be a (commutative) DGC. For each $T\in\mathscr{PT}_k$, we define a linear map $\Delta_T\colon
V\to V^{\otimes k}$ as follows: label the root by $i$, each internal edge  by $K$, each internal vertex  by $\Delta$, and each leaf by $p$. Then, $\Delta_T$ is defined as the composition of the different labels moving up from the root to the leaves. For instance,
the tree $T\in\mathscr{PT}_4$
$$
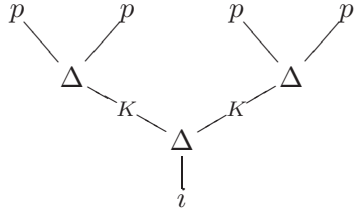
\xymatrixcolsep{1pc}
\xymatrixrowsep{1pc}
\entrymodifiers={=<1pc>} \xymatrix{
*{p}\ar@{-}[dr] & *{} & *{p}\ar@{-}[dl] & *{} & *{p}\ar@{-}[dr] & & *{p} \ar@{-}[dl]\\
*{} & {\Delta} \ar@{-}[drr]|K & *{} & *{} & *{} & \Delta\ar@{-}[dll]|K & *{} \\
*{} & *{} & *{} & \Delta\ar@{-}[d] & *{} & *{} & *{} \\
*{} & *{} & *{} & *{i} & *{} & *{} & *{} \\
}
$$
yields the map $\Delta_T=p^{\otimes 4}\circ(\Delta\circ
K\otimes\Delta\circ K)\circ \Delta\circ i\colon V\to V^{\otimes4}
$. By Theorem \ref{HTT} the transferred (commutative) $A_\infty$-coalgebra structure   in $V$ is given by $\{\Delta_k\}_{k\ge 1}$, where $\Delta_1=d$ and, for $k\ge 2$,
\begin{equation}\label{cinfinito}
\Delta_k= \sum_{T\in \mathscr{PT}_k}\Delta_T.
\end{equation}
Dually, if $M=(A,d,m)$ is a (commutative) DGA ($m$ denotes its multiplication), for each $T\in\mathscr{PT}_k$, we define a linear map $m_T\colon
V^{\otimes k}\to V$ as follows: label the root by $p$, each internal edge  by $K$,  each internal vertex  by $m$, and each leaf by $i$. Then, $m_T$ is defined as the composition of the different labels moving down from the leaves to the root. Now, the same tree above
$$
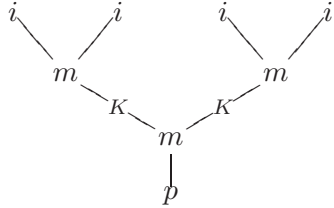
\xymatrixcolsep{1pc}
\xymatrixrowsep{1pc}
\entrymodifiers={=<1pc>} \xymatrix{
*{i}\ar@{-}[dr] & *{} & *{i}\ar@{-}[dl] & *{} & *{i}\ar@{-}[dr] & & *{i} \ar@{-}[dl]\\
*{} & {m} \ar@{-}[drr]|K & *{} & *{} & *{} & m\ar@{-}[dll]|K & *{} \\
*{} & *{} & *{} & m\ar@{-}[d] & *{} & *{} & *{} \\
*{} & *{} & *{} & *{p} & *{} & *{} & *{} \\
}
$$
produces $m_T=p\circ m\circ (K\circ m\otimes K\circ m)\circ i^{\otimes 4}\colon V^{\otimes 4}\to V
$. The transferred (commutative) $A_\infty$-algebra structure   in $V$ provided by Theorem \ref{HTT} is given by $\{m_k\}_{k\ge 1}$, where $m_1=d$ and, for $k\ge 2$,
\begin{equation}\label{ainfinito}
m_k= \sum_{T\in \mathscr{PT}_k}m_T.
\end{equation}

The  case in which $M=(L,\partial,[\,\,,\,])$ is a DGL is slightly different. For each $T$ in $\mathscr{T}_k$ define
 a
linear map
$
\ell_T\colon V^{\otimes k}\longrightarrow V
$ as follows: choose a planar embedding of $T$, label the root by $p$,
 each internal edge  by $K$, each internal vertex  by $[\,\,,\,]$, and each leaf by $i$. Then,  the map $
\widetilde\ell_T\colon V^{\otimes k}\longrightarrow V
$  is defined as the composition of the different labels moving down from the leaves to the root. For instance the planar embedding
$$
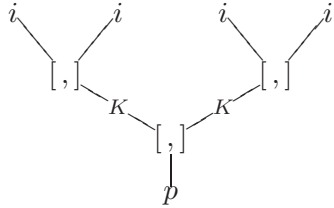
\xymatrixcolsep{1pc}
\xymatrixrowsep{1pc}
\entrymodifiers={=<1pc>} \xymatrix{
*{i}\ar@{-}[dr] & *{} & *{i}\ar@{-}[dl] & *{} & *{i}\ar@{-}[dr] & & *{i} \ar@{-}[dl]\\
*{} & {[\,,]} \ar@{-}[drr]|K & *{} & *{} & *{} & [\,,]\ar@{-}[dll]|K & *{} \\
*{} & *{} & *{} & [\,,]\ar@{-}[d] & *{} & *{} & *{} \\
*{} & *{} & *{} & *{p} & *{} & *{} & *{} \\
}
$$
of the corresponding $T\in\mathscr{T}_4$
produces the map
$$
p\circ [\,,]\circ(K\circ [\,,]\otimes K\circ [\,,])\circ i^{\otimes 4}.
$$
Define
$$
\ell_T=\widetilde\ell_{{T}}\circ \mathcal{S}_k$$
 where
 $$
 \mathcal{S}_k\colon V^{\otimes k} \to V^{\otimes k},\quad \mathcal{S}_k( a_1\otimes\cdots\otimes a_k)= \sum_{\sigma\in S_k}\varepsilon_{\sigma}\varepsilon\, a_{\sigma(1)}\otimes\cdots\otimes a_{\sigma(k)},
$$
is the symmetrization map,
in which $\varepsilon_{\sigma}$
denotes the signature of the permutation and $\varepsilon$ is the sign
given by the Koszul convention. The map $\ell_T$ is independent of the chosen planar embedding and,  by Theorem \ref{HTT}, the transferred $L_\infty$-algebra structure   on $V$ is given by $\{\ell_k\}_{k\ge 1}$, where $\ell_1=d$ and, for $k\ge 2$,
\begin{equation}\label{linfinito}
\ell_k= \sum_{T\in \mathscr{T}_k}\frac{\ell_T}{|\Aut T|}
\end{equation}
where $\Aut T$ is the automorphism group of the tree $T$.

\section{The cosimplicial cDGL $\lasu_\bullet$ and the infinity Quillen functor} \label{Quillen}

This section is devoted to describe an alternative construction for the cosimplicial cDGL $\lasu_\bullet$. The approach followed relies on a Lie bracket formulae for Lie polynomials on a general tensor algebra, that also allows us to extend the
 Quillen $\Qfuntor$ functor (see \cite[\S 22]{FHT}) for $C_\infty$-coalgebras.

Let $\varphi\colon V\to V\otimes V$  be a linear map whose image is a Lie polynomial, i.e., $\text{Im}\varphi
 \subset \mathbb{L}^2(V)$ and let $T\in \mathscr{PT}_k$. Define a linear map
$$
\varphi_T\colon V\to V^{\otimes k}
$$
recursively as follows $\varphi_T\colon V\to V^{\otimes k}$ by $\varphi_|=\id_V$ for the trivial tree $|$, and if $T$ is of the form

\begin{equation}\label{arbol}
 \xymatrixcolsep{1pc}
\xymatrixrowsep{1pc}
\entrymodifiers={=<1pc>} \xymatrix{
T'\ar@{-}[rd]&&\ar@{-}[ld]T''\\
&*{}\ar@{-}[d]&\\
&&
}\end{equation}
denoted in the sequel by $T'\smalltree T''$,
$$
\varphi_T= \Bigl( \varphi_{T'}\otimes \varphi_{T''}\Bigr) \circ \varphi .
$$
In particular, we have $\varphi_{\smalltree}=\varphi$ with this notation.

Finally, given $\psi\colon V\to W$ a homogeneous linear map we set
\begin{equation}\label{polinomio}
P_T\colon V\to W^{\otimes k},\qquad P_T=\psi^{\otimes k}\circ\varphi_T.
\end{equation}
We also consider the  linear map
\begin{equation}\label{Q}
\bracket_T\colon W^{\otimes k}\to \lib^k(W)
 \end{equation}
 defined as the nested commutator bracket moving down from the leaves to the root of the tree $T$. For example, the tree
$$
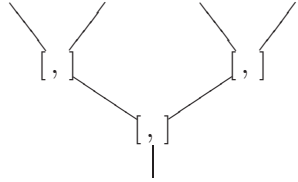
\xymatrixcolsep{1pc}
\xymatrixrowsep{1pc}
\entrymodifiers={=<1pc>} \xymatrix{
*{}\ar@{-}[dr] & *{} & *{}\ar@{-}[dl] & *{} & *{}\ar@{-}[dr] & & *{} \ar@{-}[dl]\\
*{} & {[\,,\,]} \ar@{-}[drr] & *{} & *{} & *{} &[\,,\,]\ar@{-}[dll] & *{} \\
*{} & *{} & *{} & [\,,\,]\ar@{-}[d] & *{} & *{} & *{} \\
*{} & *{} & *{} & *{_{\stackrel{}{}}} & *{} & *{} & *{} \\
}
$$
produces
$
\bracket_T=[\,,\,]\circ([\,,\,]\otimes  [\,,\,])\colon W^{\otimes 4}\to \lib^4(W)
$.

Now, for each  $T\in \mathscr{PT}_k$ we consider the subset $\overline T\subset \mathscr{PT}_k$ consisting of all the planar embeddings of trees which are isomorphic to T as non planar trees. We prove:

\begin{theorem}\label{Liepolynomial} For each $T\in \mathscr{PT}_k$,
$$
\frac{1}{|\Aut T|}\bracket_{T}\circ P_{T} =\sum_{ S\in\overline T}P_S.
$$
\end{theorem}
Here $\Aut T$ denotes the automorphism group of $T$ as a non planar rooted tree.
\begin{proof}
We proceed  by induction on the number of leaves.
Given $v\in V$, write $\varphi (v)=\sum_ia_i\otimes b_i\pm b_i\otimes a_i=\sum_i[a_i,b_i]$. Then,
$$
\begin{aligned}
P_{\smalltree}(v)&=\sum_i[\psi(a_i),\psi(b_i)]=\sum_i\frac{1}{2} \bigl( [\psi(a_i), \psi(b_i)]\pm[\psi(b_i),\psi( a_i)]\bigr)\\
&=\frac{1}{2}Q_{\smalltree}\circ P_{\smalltree}(v)
\end{aligned}
$$
which proves the statement for $k=2$.

For the general case, and to avoid excessive notation, we write $\varphi(v)=ab-(-1)^{|a||b|}ba$ and consider $\psi$ to be $\id_V$.  As in (\ref{arbol}) decompose any given tree $T\in\mathscr{PT}_k$ in the form $T=T'\smalltree T''$.

Suppose first that  $T'\not=T''$. Then,
$\overline T=\{S'\smalltree S'',\, S''\smalltree S'\}_{S'\in \overline T',S''\in\overline T''}$ and $|\Aut T|=|\Aut T'||\Aut T''|$.
Therefore,
\begin{align*}
\sum_{S\in \overline T}P_S(v)&=\Bigl( \sum_{S'\in\overline T'}P_{S'}\otimes\sum_{S''\in\overline T''}P_{S''}+ \sum_{S''\in\overline T''}P_{S''}\otimes\sum_{S'\in\overline T'}P_{S'}\Bigr)(ab-(-1)^{|a||b|}ba)\\
&=\sum_{S'\in\overline T'}P_{S'}(a)\otimes \sum_{S''\in\overline T''}P_{S''}(b)-(-1)^{|a||b|}\sum_{S'\in\overline T'}P_{S'}(b)\otimes \sum_{S''\in\overline T''}P_{S''}(a)\\
&+\sum_{S''\in\overline T''}P_{S''}(a)\otimes \sum_{S'\in\overline T'}P_{S'}(b)-(-1)^{|a||b|}\sum_{S''\in\overline T''}P_{S''}(b)\otimes \sum_{S'\in\overline T'}P_{S'}(a)\\
&=\Bigl[\sum_{S'\in\overline T'}P_{S'}(a), \sum_{S''\in\overline T''}P_{S''}(b)\Bigr]-(-1)^{|a||b|}\Bigl[\sum_{S'\in\overline T'}P_{S'}(b), \sum_{S''\in\overline T''}P_{S''}(a)\Bigr]\\
&=\Bigl[\frac{1}{|\Aut T'|}\bracket_{T'}\circ P_{T'}(a),\frac{1}{|\Aut T''|}\bracket_{T''}\circ P_{T''}(b) \Bigr]\\
&-(-1)^{|a||b|}\Bigl[\frac{1}{|\Aut T'|}\bracket_{T'}\circ P_{T'}(b),\frac{1}{|\Aut T''|}\bracket_{T''}\circ P_{T''}(a) \Bigr]\\
&=\frac{1}{|\Aut T|}\bracket_T\circ P_T (v)
\end{align*}

Assume now $T=R\smalltree R$. In this case  $\overline T=\{S'\smalltree S''\}_{S',S''\in\overline R}$  and $|\Aut T|=2|\Aut R|^2$. Then write,
\begin{align*}
\sum_{S\in \overline T}P_S(v)
&=\Bigl( \sum_{S'\in R}P_{S'}\otimes\sum_{S''\in R}P_{S''}\Bigr) (ab-(-1)^{|a||b|}ba)\\
&=\frac{1}{2}\Bigl( \sum_{S'\in R}P_{S'}\otimes\sum_{S''\in R}P_{S''}+ \sum_{S''\in R}P_{S''}\otimes\sum_{S'\in R}P_{S'}\Bigr)(ab-(-1)^{|a||b|}ba),
\end{align*}
which coincides with  $\frac{1}{|\Aut T|}\bracket_T\circ P_T (v)$ by the same computation of the previous case.
\end{proof}

Recall that an $A_\infty$-coalgebra structure on a graded vector space $V$ corresponds univocally to a differential in the complete tensor algebra $\widehat T(s^{-1}V)=\Pi_{n\ge 0}T^n(s^{-1}V)$  on the desuspension of $V$, $(s^{-1}V)_p=V_{p+1}$. Indeed,  such a differential $d$   is determined by its image  on $s^{-1}V$, which is written as a sum $d=\sum_{k\ge 1}d_k$, with $d_k(s^{-1}C)\subset T^{ k}(s^{-1}V)$, for $k\ge 1$. Then, the operators $\{\Delta_k\}_{k\ge 1}$ and $\{d_k\}_{k\ge 1}$ define each other via
\begin{equation}\label{diferencial}
\begin{aligned}
\Delta_k&=-s^{\otimes k}\circ d_k\circ s^{-1}\colon V\to
V^{\otimes k},\\
d_{k}&=-(-1)^{\frac{k(k-1)}{2}}(s^{-1})^{\otimes k}\circ\Delta_k\circ s\colon s^{-1}V\to T^{ k}(s^{-1}V).
\end{aligned}
\end{equation}
Moreover,  from Theorem \ref{Liepolynomial} we easily deduce a short proof of a long standing fact whose history is described in \cite[Introduction]{lawsu}.
\begin{theorem}\label{Quillenconstruction}
Let $(C, V, i, p, K)$ be a homotopy retract where $C$ is a cocommutative differential graded coalgebra and consider the $C_\infty$-coalgebra structure transferred on $V$. Then, the differential on the tensor algebra  restricts to Lie polynomials, i.e., $ds^{-1}v\in \libc(s^{-1}V)$ for any generator $s^{-1}v\in s^{-1}V$.
\end{theorem}
\begin{proof}
Simply observe that each $\Delta_T$  in the expression of $\Delta_k$ in (\ref{cinfinito}) is of the form $P_T\circ i$ where $P_T\colon C\to V^{\otimes k}$ is the map in (\ref{polinomio}) associated to $\psi=p$ and $\varphi= (K\otimes K)\circ\Delta$,  which is a Lie polynomial since $C$ is cocommutative. Apply now Theorem \ref{Liepolynomial} and the result follows.
\end{proof}

In particular, the differential $d$ makes $\libc(s^{-1}V)$ a cDGL which we call the {\em $\infty$-Quillen functor} and denote it  by $\Qfuntor(V)$.

\begin{remark} \label{comentario}
Note that, altough we have considered $C_\infty$-coalgebra structures transferred in Theorem \ref{Quillenconstruction} to deduce directly that $ds^{-1}v$  is a Lie polynomial, Theorem \ref{Liepolynomial} does not require the existence of transferred structures to have Lie polynomials, but merely a ``tree shaped'' form for the linear maps involved.
\end{remark}

For each $n\ge 1$ let $\Delta^n$ be the standard $n$-simplex and write in the same way the graded vector space of its  simplicial chains. We denote by $a_{i_0\dots i_k}$, $0\le i_0<\dots< i_k\le n$, the generator of the corresponding $k$-simplex. Then, by \cite[Theorems 2.3 and 2.8]{pri}, there is a unique (up to isomorphism) cDGL of the form  $\lasu_n=(\libc(s^{-1}\Delta^n),\partial)$ such that:
 \smallskip

 (1) For each $i=0,\dots,n$, the generators $s^{-1}a_i\in s^{-1}\Delta^n_0$ corresponding to vertices are Maurer-Cartan elements, $\partial s^{-1}a_i=-\frac{1}{2}[s^{-1}a_i,s^{-1}a_i]$.
\smallskip

 (2) The linear part $\partial_1$ of $\partial$ is precisely  the differential of the simplicial chain complex  $s^{-1}\Delta^n$.

Moreover \cite[\S3]{pri}, there is a natural cosimplicial cDGL structure on $\lasu_\bullet$ and we define the realization of any cDGL $L$ as the simplicial set,
$$
\langle L\rangle={\Hom_{\cdgl}}(\lasu_\bullet,L).
$$

The infinity Quillen functor allows us to define this cosimplicial cDGL $\lasu_\bullet$ in an alternative way than the constructive given in \cite[\S3]{pri} or \cite[\S7]{pri} in terms of Lie coalgebras or the $\mathfrak{mc}_\bullet$ of \cite[\S5.2]{ni1}.

Denote by $\asu_\bullet$
 the simplicial CDGA of PL-forms on the standard simplices,
$$
\asu_n=\Lambda(t_0,\dots,t_n,dt_0,\dots,dt_n)/(\textstyle{\sum t_i-1,\sum dt_i)},
$$
 and let  $C^*(\Delta^\bullet)$   be the simplicial cochain complex also on the standard simplices. Then \cite{du,du2,getz}, there is a homotopy retract
\begin{equation}\label{retract}
\xymatrix{
 \ar@(ul,dl)@<-5.5ex>[]_{K_\bullet}
 &
  \asu_\bullet \ar@<0.75ex>[r]^-{p_\bullet}
  &
   {C^*(\Delta^\bullet),} \ar@<0.75ex>[l]^-{i_\bullet} }
\end{equation}
where the maps $p_{\bullet}$ and $i_{\bullet}$ are defined as follows:

Let  $\alpha_{i_0\dots i_k}$ be the basis for $C^*(\Delta^n)$ defined by
$$\langle \alpha_{i_0\dots i_k}, a_{j_0\dots j_k}\rangle = \left\{\renewcommand{\arraystretch}{1.5}\begin{array}{ll}
(-1)^{\frac{k(k-1)}{2}} \, & \mbox{if }(j_0, \dots , j_k) = (i_0,\dots , i_k),\\
\qquad 0 & \mbox{otherwise.}
\end{array}\renewcommand{\arraystretch}{1}
\right.$$
Then,  $
 i_n(\alpha_{i_0\dots i_k}) $ is the Whitney elementary form
 $\omega_{i_0\dots i_k} $ defined by
 $$
 \omega_{i_0\dots i_k}  = k! \sum_{j=0}^k (-1)^j t_{i_j} dt_{i_0}\cdots \widehat{dt_{ij}}\cdots dt_{i_k}.$$
The map $p_{n}\colon \asu_n\to C^*(\Delta^n)$ is defined by
$$p_n(\omega) = \sum_{k=0}^n \sum_{i_0<\dots < i_k} \alpha_{i_0\dots i_k} {\mathcal I}_{i_0\dots i_k}(\omega),$$
with
$$ {\mathcal I}_{i_0\dots i_k}(t^{b_1}_{i_1}\dots t_{i_k}^{b_k} dt_{i_1}\dots dt_{i_k})= \frac{b_1!\cdots b_k!}{(b_1+\dots + b_k+k)!},$$
 and $0$ otherwise.
 In particular, $ {\mathcal I}_{i_0\dots i_k}( \omega_{i_0\dots i_k} )=1$.

 Theorem \ref{HTT} induces a simplicial $C_\infty$-algebra structure on $C^*(\Delta^\bullet)$ and, since this is a finite dimensional simplicial cochain complex, Remark \ref{remarco1} provides by dualizing a $C_\infty$-coalgebra structure on the simplicial chain complex $\Delta^\bullet$ and a cosimplicial cDGL of the form $\libc(s^{-1}\Delta^\bullet)$ applying Theorem \ref{Liepolynomial}. By uniqueness \cite[Theorem 2.8]{pri} it follows that $\libc(s^{-1}\Delta^\bullet)=\lasu_\bullet$

\section{Maurer-Cartan elements and realization of Lie algebras}\label{realization}

In this section we will show that the realization $\langle \lasu \rangle_\bullet$ is a strong deformation retract of the Deligne-Hinich-Getzler $\infty$-groupoid $\text{MC} (L\otimes \asu_\bullet )$ relying on explicit combinatorial computations.

 Let $C$ be a CDGC and let $L$ be a DGL. Recall that the linear maps $\Hom(C, L)$ have a natural DGL structure  with the usual differential given by the bracket $[\partial , \delta ]$ of the differentials in $L$ and $C$ respectively, and the {\em convolution Lie bracket} $[f,g]=[\,,\,]\circ (f\otimes g)\circ \Delta$.

Hence, if $(C, V, i, p, K)$ is   a homotopy retract,   then
$$(\Hom(C, L), \Hom(V, L), i^*, p^*, K^*)$$
is again a homotopy retract which, by Theorem \ref{HTT}, induces an $L_\infty$-structure on $\Hom(V,L)$.

We now prove that, if $L$ is complete, then the restriction,
 $$
{\Hom_{\cdgl}}(\Qfuntor(V),L)\longrightarrow \Hom_{-1}(V,L),\quad f\mapsto f\circ s^{-1},
 $$
 of any cDGL morphism  to its generators provides the following bijection.

\begin{theorem}\label{main} For any $\cdgl$ $L$, ${\Hom_{\cdgl}}(\Qfuntor(V),L)\cong\mc\bigl(\Hom(V,L)\bigr)$.
\end{theorem}

For $k\ge 2$, let $\Delta_k$ and $\ell_k$ denote, as in (\ref{cinfinito}) and (\ref{linfinito}),  the $k$-th diagonal and $k$-th bracket induced on $V$ and $\Hom(V,L)$ respectively. A straightforward computation provides the following.

\begin{lemma}\label{lemalinfinito}
 For each $T\in \mathscr{PT}_k$,
$$
\widetilde\ell_T=\gamma\circ Q_T\circ (-\otimes\stackrel{k}{ \cdots }\otimes -)\circ \Delta_T.
$$\hfill $\square$
\end{lemma}
Here $Q_T$ is the map (\ref{Q}) and $\gamma\colon \lib^k(L)\to L$ is the Lie bracketing morphism induced by the identity $\id_L$.

\begin{proof}[Proof of Theorem $\ref{main}$]

We have to show that $f\in{\Hom_{\cdgl}}(\Qfuntor(V),L)$ if and only if $fs^{-1}\in\mc\bigl(\Hom(V,L)\bigr)$. In other words,
$$
\partial f(s^{-1}v)+\sum_{k\ge 1}f(d_ks^{-1}v)=0\,\,\text{if and only if}\,\,\sum_{k\ge 1}\frac{1}{k!}\ell_k(fs^{-1},{\ldots},fs^{-1})(v)=0,
$$
for any $v\in V$, being $d=\sum_{k\ge 1}d_k$
 as in (\ref{diferencial}).
Recall that $\ell_1(fs^{-1})$ is precisely the differential in $\Hom(V,L)$, i.e., $\ell_1(fs^{-1})(v)=\partial f(s^{-1}v)+f(d_1s^{-1}v)$, and therefore it suffices to show that for any $v\in V$,
$$
\sum_{k\ge 2}f(d_ks^{-1}v)=0\quad\text{if and only if}\quad \sum_{k\ge 2}\frac{1}{k!}\ell_k(fs^{-1},{\ldots},fs^{-1})(v)=0.
$$
 In fact we prove that, as maps,
 $$
 fd_ks^{-1}=-\frac{1}{k!}\ell_k(fs^{-1},{\ldots},fs^{-1}),\quad\text{for any $k\ge 2$.}
 $$
On the one hand, using formula (\ref{linfinito}) first and then Lemma \ref{lemalinfinito},
$$
\begin{aligned}
\frac{1}{k!}\ell_k(fs^{-1},{\ldots},fs^{-1})&=\frac{1}{k!} \sum_{T\in \mathscr{T}_k}\frac{\ell_T(fs^{-1},{\ldots},fs^{-1})}{|\Aut T|}\\
&= \sum_{T\in \mathscr{T}_k}\frac{\widetilde\ell_T(fs^{-1},{\ldots},fs^{-1})}{|\Aut T|}\\
&=\sum_{T\in \mathscr{T}_k}\frac{\gamma\circ Q_T\circ (fs^{-1})^{\otimes k}\circ \Delta_T}{|\Aut T|}.\\
\end{aligned}
$$
Now observe that we may write
$$
(fs^{-1})^{\otimes k}\circ \Delta_T=P_T\circ i$$ where $P_T\colon C\to L^{\otimes k}$ is the map in (\ref{polinomio}) associated to $\psi=fs^{-1}p$ and $\varphi= (K\otimes K)\circ\Delta$  which is a Lie polynomial since $C$ is cocommutative.
Hence, by Theorem \ref{Liepolynomial}, the above reduces to
$$
 \sum_{T\in \mathscr{PT}_k}\gamma\circ (fs^{-1})^{\otimes k}\circ \Delta_T.
$$
On the other hand, in view of equation (\ref{diferencial}), formula (\ref{cinfinito}), and using that $f$ commutes with Lie brackets, we deduce:
\begin{equation}\label{util}
\begin{aligned}
fd_ks^{-1}&=-(-1)^{\frac{k(k-1)}{2}}f\circ (s^{-1})^{\otimes k}\circ\Delta_k\circ s\circ s^{-1}\\
&=-(-1)^{\frac{k(k-1)}{2}}f\circ (s^{-1})^{\otimes k}\circ  \sum_{T\in \mathscr{PT}_k}\Delta_T\\
&=- \sum_{T\in \mathscr{PT}_k}\gamma\circ (fs^{-1})^{\otimes k}\circ \Delta_T.
\end{aligned}
\end{equation}
\end{proof}

In the dual setting, let $(A, V, i, p, K)$ be   a homotopy retract of the  CDGA $A$ and denote by $\{m_k\}_{k\ge 1}$ the $C_\infty$-structure induced on $V$ by Theorem \ref{HTT}. If $L$ is a cDGL consider
$$
A\cotimes L=\varprojlim_n (A\otimes L/L^n)$$
with the naturally  inherited DGL structure, and observe that
$$(A\cotimes L, V\cotimes L, i\cotimes \id_L, p\cotimes \id_L, K\cotimes \id_L)$$
is again a homotopy retract which induces an $L_\infty$-structure on $V\cotimes L$.

\begin{remark}\label{remarco1} To identify its Maurer-Cartan set as cDGL morphisms it is necessary to assume that $V$ is a finite type vector space and we do so henceforth. Under this assumption we have:

\medskip

(i) The natural maps
$$
V\cotimes L=\varprojlim_n (V\otimes L/L^n)\stackrel{\cong}{\longrightarrow}V\otimes \varprojlim_n L/L^n=V\otimes L,$$
$$
\Psi\colon V\otimes L\stackrel{\cong}{\longrightarrow} \Hom(V^\sharp,L),\quad \Psi(a\otimes x)(\alpha)=(-1)^{|a||x|}\langle a,\alpha\rangle x,
$$
are isomorphisms  of graded vector spaces.

\medskip

(ii) For each $k\ge 1$, the dual map $
m_k^\sharp\colon V^\sharp\longrightarrow (V^{\otimes k})^{\sharp}
$, once composed with the canonical isomorphism  $(V^{\otimes k})^{\sharp}\cong {V^{\sharp}}^{\otimes k}$,
defines a $C_\infty$-coalgebra on $V^\sharp$ which, by Theorem \ref{Liepolynomial}, produces the cDGL $\Qfuntor (V^\sharp)=\libc(s^{-1}V^\sharp)$.
\end{remark}

We show that  if $L$ is complete,  then the restriction,
 $$
 {\Hom_{\cdgl}}(\Qfuntor(V^{\sharp}),L)\longrightarrow \Hom_{-1}(V^{\sharp},L)\cong (V\otimes L)_{-1},
 $$
 of any DGL morphism  to its generators provides the following bijection.

\begin{theorem}\label{main2} For any $\cdgl$ $L$, ${\Hom_{\cdgl}}(\Qfuntor(V^\sharp),L)\cong\mc(V\otimes L)$.
\end{theorem}

For $k\ge 2$, let $\ell_k$ denote, as in (\ref{linfinito}),  the $k$-th bracket induced on $V\otimes L$. Then,  the proof of the following analogue of Lemma \ref{lemalinfinito} is also straightforward.

\begin{lemma}\label{lemaainfinito}
 For each $T\in \mathscr{PT}_k$,
$$
\widetilde\ell_T(a_1\otimes x_1,\dots,a_k\otimes x_k)=m_T(a_1\otimes\cdots\otimes a_k)\otimes \gamma\circ Q_T(x_1\otimes\cdots \otimes x_k).
$$\hfill $\square$
\end{lemma}

\begin{proof}[Proof of Theorem $\ref{main2}$]

We have to prove that $f\in{\Hom_{\cdgl}}(\Qfuntor(V),L)$ if and only if $fs^{-1}_{|_V}=\Psi(z)$ with $z\in\mc(V\otimes L)$.

For it, let $z_i\in V\otimes L$ and let $\Psi(z_i)=g_i\in \Hom(V^\sharp,L)$,  $i=1,\dots,k$ with $k\ge 2$. Then,  it is a straightforward computation using Lemma \ref{lemaainfinito} to show that, for each $T\in \mathscr{PT}_k$,
$$
\Psi\bigl(\widetilde\ell_T(z_1,\dots,z_k)\bigr)=\gamma\circ Q_T\circ(g_1\otimes\dots\otimes g_k)\circ m_T^\sharp.
$$
In particular,  choosing $z_i=z$ for all $i$, with $\Psi(z)=g$, and applying formula (\ref{linfinito}), we get:
$$
\Psi\bigl(\frac{1}{k!}\ell_k(z,\dots,z)\bigr)=\sum_{T\in \mathscr{T}_k}\frac{\gamma\circ Q_T\circ g^{\otimes k}\circ m_T^\sharp}{|\Aut T|},\quad k\ge 2,
$$
while, for $k=1$,         $\Psi(\ell_1z)$  is the usual differential $\partial g-(-1)^{|g|} gd_1$ on $\Hom(V^\sharp,L)$.

Hence, we need to prove that given $f\colon \Qfuntor(V)\to L$ of degree zero and $v\in V$, the following two equations are equivalent:
$$
\begin{aligned}
&\partial f(s^{-1}v)+\sum_{k\ge 1}f(d_ks^{-1}v)=0,\\
&
\partial f(s^{-1}v)+f(d_1s^{-1}v)+\sum_{k\ge 2}\sum_{T\in \mathscr{T}_k}\frac{\gamma\circ Q_T\circ (fs^{-1})^{\otimes k}\circ m_T^\sharp}{|\Aut T|}(v)=0.
\end{aligned}
$$
That is,
$$
\sum_{k\ge 2}f(d_ks^{-1}v)=0\quad\text{if and only if}\quad \sum_{k\ge 2}\sum_{T\in \mathscr{T}_k}\frac{\gamma\circ Q_T\circ (fs^{-1})^{\otimes k}\circ m_T^\sharp}{|\Aut T|}(v)=0.
$$
We finish by checking that, in fact,
$$
fd_ks^{-1}= -\sum_{T\in \mathscr{T}_k}\frac{\gamma\circ Q_T\circ (fs^{-1})^{\otimes k}\circ m_T^\sharp}{|\Aut T|},\quad\text{for any $k\ge 2$.}
$$
On the one hand,   $(fs^{-1})^{\otimes k}\circ m_T^\sharp$ is a map of the form  (\ref{polinomio}) since  $V^\sharp$ is a $C_\infty$-coalgebra.
Hence, by Theorem \ref{Liepolynomial},
$$
\sum_{T\in \mathscr{T}_k}\frac{\gamma\circ Q_T\circ (fs^{-1})^{\otimes k}\circ m_T^\sharp}{|\Aut T|}=  \sum_{T\in \mathscr{PT}_k}\gamma\circ (fs^{-1})^{\otimes k}\circ m_T^\sharp.
$$
On the other hand, the computation in (\ref{util}) replacing $\Delta_T$ by $m_T^{\sharp}$, provides
$$
fd_ks^{-1}=-\sum_{T\in \mathscr{PT}_k}\gamma\circ (fs^{-1})^{\otimes k}\circ m_T^\sharp
$$
\end{proof}

Next, let $L$ be a cDGL and consider the simplicial homotopy retract
\begin{equation}\label{retractoL}
\xymatrix{
 \ar@(ul,dl)@<-5.5ex>[]_{K_\bullet\cotimes\id_L}
 &
  \asu_\bullet\cotimes L \ar@<0.75ex>[r]^-{p_\bullet\cotimes\id_L}
  &
   {C^*(\Delta^\bullet){\otimes} L} \ar@<0.75ex>[l]^-{i_\bullet\cotimes\id_L} }
\end{equation}
induced by $(\ref{retract})$ which, via Theorem \ref{HTT}, provides quasi-isomorphisms of $L_\infty$-algebras,
$$
\xymatrix{
  \asu_\bullet\cotimes L \ar@<0.75ex>[r]^-{P_\bullet}
  &
   {C^*(\Delta^\bullet){\otimes} L} \ar@<0.75ex>[l]^-{I_\bullet}, }
$$
and in particular, simplicial set maps which were proved to be weak homotopy equivalences, first for nilpotent DGL's in \cite{getz}, and then for any ``strict'' $L_\infty$-algebra in \cite[Theorem 5.2.24]{Bandiera},
$$
\xymatrix{
  \mc(\asu_\bullet\cotimes L) \ar@<0.75ex>[r]^-{\mc(P_\bullet)}
  &
   \mc({C^*(\Delta^\bullet){\otimes} L)} \ar@<0.75ex>[l]^-{\mc(I_\bullet)}. }
$$
Then, we have (cf. \cite[Lemma 5.2.26]{Bandiera}):
\begin{corollary}\label{equivalencia} Given any $\cdgl$ $L$,
$$
\mc(C^*(\Delta^\bullet){\otimes} L)\cong \langle L\rangle
 $$
 as simplicial sets.
 \end{corollary}
 \begin{proof}
 Simply apply Theorem \ref{main2} to identify the simplicial sets
 $$
\mc(C^*(\Delta^\bullet){\otimes} L)\cong {\Hom_{\cdgl}}(\lasu_\bullet,L).
$$
\end{proof}
 Therefore, we have simplicial maps,
\begin{equation}\label{buenomc}
\xymatrix{
  \mc(\asu_\bullet\cotimes L) \ar@<0.75ex>[r]^-{\mc(P_\bullet)}
  &
   \langle L\rangle \ar@<0.75ex>[l]^-{\mc(I_\bullet)}, }
\end{equation}
relating our realization with the Hinich ``contents or nerve" of $L$ \cite{hi}.  The following is Theorem \ref{principal}.

\begin{theorem}\label{principal2} The maps $\mc(I_\bullet)$ and $\mc(P_\bullet)$ are homotopy equivalences which make  $\langle L\rangle$ a strong deformation retract of $\mc(\asu_\bullet{\cotimes} L)$.
\end{theorem}

 \begin{proof}Theorem \ref{HTT} provides $\mc(P_\bullet)\mc( I_\bullet)=\id_{\langle L\rangle}$. Also note that both $\langle L\rangle$ and $\mc(\asu_\bullet{\cotimes} L)$ are Kan complexes, see \cite[Proposition 7.3]{se} and \cite{getz} respectively. We then finish by showing that $\mc(I_\bullet)$ of (\ref{buenomc}) is a weak homotopy equivalence.

 On the one hand, both $\pi_0\langle L\rangle$ and  $\pi_0\mc(\asu_\bullet\cotimes L)$ coincide with the set $\widetilde{\mc}(L)$ of Maurer-Cartan elements of $L$ modulo the gauge relation, see \cite[Proposition 4.4]{pri} and \cite[Introduction]{getz} respectively. Now, for any $z\in\mc(L)$ consider the ``localization'' of $L$ at $z$ which is the cDGL
 $$
 L^{(z)}=(L,\partial_z)/(L_{<0}\oplus M)
 $$
 in which $\partial_z=\partial+\ad_z$ is the original differential on $L$ perturbed by the adjoint on $z$, and $M$ is  a complement of $\ker \partial_z$ in $L_0$. Then \cite[Proposition 4.6]{pri}, the path component $\langle L\rangle_z$ of  $\langle L\rangle$ at $z\in\widetilde{\mc}(L)$ is precisely $\langle L^{(z)}\rangle$. Moreover \cite[Proposition 4.5]{pri}, and this is valid for any non-negatively graded cDGL,  for any $n\ge 1$, there is a natural group isomorphism,
 $$
 \varphi\colon H_{n-1}(L^{(z)})\stackrel{\cong}{\longrightarrow}\pi_n \langle L^{(z)}\rangle,\qquad \varphi[\Phi]=\overline f,
 $$
  where this is the homotopy class represented by the the cDGL morphism $f\colon \libc(s^{-1}\Delta^n)\to L^{(z)}$ which takes the top generator $s^{-1}a_{0\dots n}$   to $\Phi$, and is zero on the rest of generators. For completeness, we remark that for $n=1$ the group structure on $H_0(L^{(z)})$ is given by the Baker-Campbell-Hausdorff formula. Composing $\varphi$ with the morphism induced in homotopy groups by
 the isomorphism in Corollary \ref{equivalencia} we obtain a group isomorphism:
 $$
 \psi\colon H_{n-1}(L^{(z)})\stackrel{\cong}{\longrightarrow}\pi_n \mc(C^*(\Delta^\bullet){\otimes} L^{(z)}),\quad \psi[\Phi]=\overline{\alpha_{0\dots n}{\otimes}\Phi}.
 $$
On the other hand, in \cite[Corollary 1.3]{ber} (cf.  \cite[Proposition 7.20]{lamar}) it is proved that the path component $\mc(\asu_\bullet{\cotimes} L)_z$ of $\mc(\asu_\bullet{\cotimes} L)$ at $z\in\widetilde{\mc}(L)$ is precisely $\mc(\asu_\bullet{\cotimes} L^{(z)})$. Moreover \cite[Theorem 1]{ber} (see \S4 of this reference for details), for each $n\ge 1$, there is a group isomorphism
$H_{n-1}(L^{(z)})\cong \pi_n\mc(\asu_\bullet{\cotimes} L^{(z)})$ which sends the homology class $[\Phi]$ to $\overline{\omega_{0\dots n}{\otimes} \Phi}$. We finish by checking that, for each $n\ge 1$, the induced map at the component determined by  $z$,
$$
\mc(I_n)\colon \mc(C^*(\Delta^n){\otimes} L^{(z)})\to \mc(\asu_n{\cotimes} L^{(z)}),
$$
satisfies $\mc(I_n)(\alpha_{0\dots n}{\otimes}\Phi)=\omega_{0\dots n}{\otimes} \Phi$.

As a general fact,
$$
\mc(I_n)(x)=\sum_k\frac{1}{k!}I_n^{(k)}(x,\dots,x)
$$
 being $\{I_n^{(k)}\}_{k\ge 1}$ the Taylor series of the $L_\infty$-morphism $I_n$.  Theorem \ref{HTT} gives an explicit recursive description of this series: $I_n^{(1)}=i_n$ and
$$
I_n^{(k)}(x_1,...,x_k)=\sum_{i=1}^{k-1}\sum_{\sigma \in \widetilde S(i, k-i)}\varepsilon(\sigma) K[I_i(x_{\sigma(1)},...,x_{\sigma(i)}),I_{k-i}(x_{\sigma(i+1)},...,x_{\sigma(k)})]
$$
in which $K$ is the chain homotopy coming from $(\ref{retract})$ and $\widetilde S(i, k-i)$ denote the shuffle permutations which fix $1$.

In our particular case, as any power of $\omega_{0\dots n}$ vanishes, and $i_n(\alpha_{0\dots n})=\omega_{0\dots n}$, it follows that
$$
I_n^{(1)}(\alpha_{0\dots n}{\otimes}\Phi)=\omega_{0\dots n}{\otimes} \Phi,\quad I_n^{(k)}(\alpha_{0\dots n}{\otimes}\Phi,\dots,\alpha_{0\dots n}{\otimes}\Phi)=0,\,\, k\ge 2,
$$
and the proof is complete.
\end{proof}

Recall that the inclusion $\gamma_\bullet(L)\hookrightarrow \mc_\bullet(\asu\cotimes L)$ of the {\em nerve} of the cDGL $L$ \cite[\S5]{getz}, defined as,
$$
\gamma_\bullet(L)=\{x\in\mc_\bullet(\asu\cotimes L),\,\,(K\widehat\otimes\id_L)_\bullet(x)=0\},
$$
is a homotopy equivalence \cite[Corollary 5.9]{getz} which makes  the nerve of $L$ a simplicial set more tractable that the Getzler-Hinich realization. Here we show,

\begin{theorem}\label{final} The realization $\langle L\rangle$ of any cDGL is isomorphic to $\gamma_\bullet(L)$ as simplicial sets. 
\end{theorem}
\begin{proof} This is an immediate consequence of Corollary \ref{equivalencia} and \cite[Theorem 1]{getz2} or  \cite[Proposition 5.2.7]{Bandiera}  which, for the particular homotopy retract in (\ref{retractoL}), reads that $ \gamma_\bullet(L)\cong  \mc(C^*(\Delta^\bullet){\otimes} L)$.
\end{proof}

{\small
\vspace{5mm}\noindent {\sc Departamento de Algebra, Geometr\'ia y Topolog\'ia,
Universidad de M\'alaga,
Ap. 59,
29080-M\'alaga,
Espa\~na}

{\it E-mail address:} ubuijs@uma.es
\\[2mm]
{\sc  Institut de Math\'ematiques et Physique,
Universit\'e Catholique de Louvain-la-Neuve,
Louvain-la-Neuve,
Belgique}

{\it E-mail address:}  Yves.felix@uclouvain.be
\\[2mm]
{\sc Departamento de Algebra, Geometr\'ia y Topolog\'ia,
Universidad de M\'alaga,
Ap. 59,
29080-M\'alaga,
Espa\~na}

{\it E-mail address:} aniceto@uma.es
\\[2mm]{\sc D\'epartement de Mathematiques,
         UMR 8524,
         Universit\'e de Lille~1,
         59655 Villeneuve d'Ascq Cedex,
         France }

         {\it E-mail address:} Daniel.tanre@univ-lille1.fr}

\end{document}